\documentclass[12pt, reqno]{amsart}
\usepackage{amssymb, mathrsfs}
\usepackage{latexsym}
\usepackage{amsmath}
\usepackage{amsthm}
\usepackage{amsfonts}
\usepackage{dsfont, upgreek}
\usepackage{mathtools}
\usepackage{epsfig}
\usepackage{amscd}
\usepackage{graphicx}
\usepackage{tikz-cd}

\usepackage{enumitem}
\setlist[enumerate]{itemsep=0.5ex}

\usepackage{color}
\usepackage{tikz}
\usetikzlibrary{patterns}

\usepackage[left=1.4in,right=1.4in,top=1.2in,bottom=1.2in]{geometry}

\usetikzlibrary{decorations.markings}
\usetikzlibrary{arrows.meta}

\usepackage{extarrows}

\usepackage{adjustbox}
\usepackage{pgfplots}
\usepgfplotslibrary{colormaps}
\usepackage{slashed}
\usepackage[msc-links, lite,non-compressed-cites]{amsrefs}

\usepackage[colorinlistoftodos,prependcaption,textsize=tiny]{todonotes}  

\usepackage[all,cmtip]{xy}
\theoremstyle{plain}
\newtheorem{theorem}{Theorem}[section]
\newtheorem{proposition}[theorem]{Proposition}
\newtheorem{lemma}[theorem]{Lemma}
\newtheorem{corollary}[theorem]{Corollary}

\newtheorem{conjecture}[theorem]{Conjecture}

\theoremstyle{definition} 
\newtheorem{definition}[theorem]{Definition}

\newtheorem{example}[theorem]{Example}
\newtheorem*{claim*}{Claim}

\theoremstyle{remark}

\numberwithin{equation}{section}

\newcommand{\Sc}{\mathrm{Sc}}

\newcommand{\C}{\mathbb{C}}

\newcommand{\R}{\mathbb{R}}

\newcommand{\ind}{\textup{Ind}}
\newcommand{\id}{\mathrm{id}}

\newcommand{\supp}{\mathrm{supp}}

\newcommand{\prop}{\mathrm{prop}}

\newcommand{\Ric}{\mathrm{Ric}}

\newcommand{\Q}{\mathbb{Q}}

\newcommand{\interior}[1]{%
	{\kern0pt#1}^{\mathrm{\,o}}%
}

\makeatletter
\let\save@mathaccent\mathaccent
\newcommand*\if@single[3]{%
	\setbox0\hbox{${\mathaccent"0362{#1}}^H$}%
	\setbox2\hbox{${\mathaccent"0362{\kern0pt#1}}^H$}%
	\ifdim\ht0=\ht2 #3\else #2\fi
}
\newcommand*\rel@kern[1]{\kern#1\dimexpr\macc@kerna}
\newcommand*\wideaccent[2]{\@ifnextchar^{{\wide@accent{#1}{#2}{0}}}{\wide@accent{#1}{#2}{1}}}
\newcommand*\wide@accent[3]{\if@single{#2}{\wide@accent@{#1}{#2}{#3}{1}}{\wide@accent@{#1}{#2}{#3}{2}}}
\newcommand*\wide@accent@[4]{%
	\begingroup
	\def\mathaccent##1##2{%
		\let\mathaccent\save@mathaccent
		\if#42 \let\macc@nucleus\first@char \fi
		\setbox\z@\hbox{$\macc@style{\macc@nucleus}_{}$}%
		\setbox\tw@\hbox{$\macc@style{\macc@nucleus}{}_{}$}%
		\dimen@\wd\tw@
		\advance\dimen@-\wd\z@
		\divide\dimen@ 3
		\@tempdima\wd\tw@
		\advance\@tempdima-\scriptspace
		\divide\@tempdima 10
		\advance\dimen@-\@tempdima
		\ifdim\dimen@>\z@ \dimen@0pt\fi
		\rel@kern{0.6}\kern-\dimen@
		\if#41
		#1{\rel@kern{-0.6}\kern\dimen@\macc@nucleus\rel@kern{0.4}\kern\dimen@}%
		\advance\dimen@0.4\dimexpr\macc@kerna
		\let\final@kern#3%
		\ifdim\dimen@<\z@ \let\final@kern1\fi
		\if\final@kern1 \kern-\dimen@\fi
		\else
		#1{\rel@kern{-0.6}\kern\dimen@#2}%
		\fi
	}%
	\macc@depth\@ne
	\let\math@bgroup\@empty \let\math@egroup\macc@set@skewchar
	\mathsurround\z@ \frozen@everymath{\mathgroup\macc@group\relax}%
	\macc@set@skewchar\relax
	\let\mathaccentV\macc@nested@a
	\if#41
	\macc@nested@a\relax111{#2}%
	\else
	\def\gobble@till@marker##1\endmarker{}%
	\futurelet\first@char\gobble@till@marker#2\endmarker
	\ifcat\noexpand\first@char A\else
	\def\first@char{}%
	\fi
	\macc@nested@a\relax111{\first@char}%
	\fi
	\endgroup
}
\makeatother

\newcommand\overbar{\wideaccent\overline}

\makeatletter
\newcommand*{\transpose}{%
	{\mathpalette\@transpose{}}%
}
\newcommand*{\@transpose}[2]{%
	\raisebox{\depth}{$\m@th#1\intercal$}%
}
\makeatother

\begin{document}
\title[spin surgeries and Novikov conjecture]{Non-negative scalar curvature on spin surgeries and Novikov conjecture}

\author{Jinmin Wang}
\address[Jinmin Wang]{State Key Laboratory of Mathematical Sciences, Academy of Mathematics and Systems Science, Chinese Academy of Sciences}
\email{jinmin@amss.ac.cn}
\thanks{The first author is partially supported by NSFC 12501169.}

\begin{abstract}
	Let $M$ be a closed aspherical manifold. Assume that the rational strong Novikov conjecture holds for $\pi_1(M)$. We show that on any spin surgery of $M$ along a region whose induced homomorphism on the fundamental group is trivial, every complete metric with non-negative scalar curvature is Ricci-flat. In particular, on the connected sum of $M$ with a spin manifold, any complete metric with non-negative scalar curvature is Ricci-flat.
\end{abstract}

\maketitle

\section{Introduction}

The study of scalar curvature on Riemannian manifolds has played a central role in differential geometry and geometric topology for several decades. A landmark result due to Gromov and Lawson \cite{GromovLawson,GromovLawsonAnnals}, using index theory, and independently to Schoen and Yau \cite{SchoenYau3,SchoenYau}, using minimal hypersurface methods, asserts that the $n$-torus does not admit any Riemannian metric of positive scalar curvature. More generally, Gromov--Lawson--Rosenberg conjectured that no aspherical manifold admits a metric of positive scalar curvature. This conjecture is known to hold for aspherical manifolds whose fundamental groups satisfy the rational strong Novikov conjecture \cite{Rosenberg,Yu}, as well as for aspherical manifolds of dimension at most five \cite{ChodoshLi,Gromov5fold}.

The non-existence of positive scalar metric holds for more generally enlargeable manifolds, introduced by Gromov--Lawson \cite{GromovLawson,GromovLawsonAnnals}, including connected sums of torus with a closed spin manifold. Using index-theoretic techniques, Su--Zhang \cite{SuZhang} established a quantitative version of this obstruction for connected sums of enlargeable manifolds with non-spin closed manifolds. Subsequently, Wang--Zhang \cite{WangZhang} and Su \cite{SuRemark} showed that the connected sum of an enlargeable manifold with any spin manifold does not admit a metric of positive scalar curvature. Related non-existence results for positive scalar curvature on non-spin connected sums, obtained via minimal hypersurface methods, can be found in \cite{ChodoshLi,ChenChuZhu,MR4291609}.

In this paper, we extend the above non-existence of positive scalar metrics to a significantly broader class of manifolds. 
\begin{theorem}\label{thm:connectedSum}
	Let $M^n$ be a closed aspherical manifold and $N$ a spin manifold. Assume that the rational strong Novikov conjecture holds for $\pi_1(M)$.
	Then the connected sum $M\# N$ does not carry any complete metric with positive scalar curvature. Furthermore, any complete metric with non-negative scalar curvature on $M\# N$ is Ricci-flat.
\end{theorem}

	We remark that in Theorem \ref{thm:connectedSum}, the manifold $N$ is allowed to be non-compact, and the metric on $M\#N$ is not required to have bounded geometry. The rational strong Novikov conjecture for $\pi_1(M)$ \cite[Section~2]{Rosenberg} asserts that the assembly map from the $K$-homology of $M$ to the $K$-theory of the group $C^*$-algebra of $\pi_1(M)$ is rationally injective; see Conjecture \ref{conj:Novikov} for a precise formulation. The reader may consult \cite{GongWuYu} for recent progress and a survey of classes of groups known to satisfy this conjecture. To date, no counterexample to the rational strong Novikov conjecture is known.

We further generalize Theorem \ref{thm:connectedSum} to spin surgeries on aspherical manifolds. 

\begin{theorem}\label{thm:main}
	Let $M^n$ be a closed aspherical manifold and $\Omega\subset M$ a smooth region. Assume that
	$$\pi_1(\Omega)\to \pi_1(M)\textup{ is the trivial map}.$$
	Let $N^n$ be a spin manifold with boundary, and $i\colon \partial N\to \partial \Omega$ a diffeomorphism preserving the spin structure\footnote{Although $M$ itself need not be spin, the triviality of $\pi_1(\Omega)\to\pi_1(M)$ implies that $\Omega$ and $\partial\Omega$ inherit a canonical spin structure from the spin universal cover $\widetilde M$.}. Set $Z\coloneqq (M\backslash\Omega)\cup_{\partial\Omega} N$. If the rational strong Novikov conjecture holds for $\pi_1(M)$, then $Z$ does not carry any complete metric with positive scalar curvature. Furthermore, any complete metric with non-negative scalar curvature on $Z$ is Ricci-flat.
\end{theorem}

Theorem \ref{thm:connectedSum} is a special case of Theorem \ref{thm:main} by choosing $\Omega$ to be a small ball in $M$. The assumption that $\pi_1(\Omega)$ maps trivially into $\pi_1(M)$ is essential; in Example \ref{example} we construct a spin surgery on a closed aspherical manifold that admits a metric of positive scalar curvature when this condition fails.

Our proofs rely on Dirac operator methods. A fundamental ingredient is the relative index theorem for higher indices due to Xie and Yu \cite[Theorem~A]{XieYu14}. Inspired by the approaches in \cite{SuZhang,WangXieLinf}, we construct an explicit representative of the relative higher index and compute it in the $K$-theory of the group $C^*$-algebra. We also provide a Dirac-operator-based proof of the Ricci-flatness statement, following ideas similar to those in \cite{WangZhu}.

As an immediate application, by taking $N=\partial\Omega\times[0,1)$ in Theorem \ref{thm:main}, we obtain the following corollary.

\begin{corollary}\label{coro:remove}
	Let $M^n$ be a closed aspherical manifold and $\Omega\subset M$ a smooth region. Assume that
	$$\pi_1(\Omega)\to \pi_1(M)\textup{ is the trivial map},$$
	and the rational strong Novikov conjecture holds for $\pi_1(M)$. Then $M\backslash\Omega$ does not carry any complete metric with positive scalar curvature. 
\end{corollary}

The paper is organized as follows. In Section \ref{sec:preliminary}, we review the higher indices and the strong Novikov conjecture. In Section \ref{sec:UPSC}, we prove a simplified version of Theorem \ref{thm:main}, establishing the non-existence of uniformly positive scalar curvature metrics on spin surgeries. The full proof of Theorem \ref{thm:main} is given in Section \ref{sec:nnsc}.

\medskip
\noindent\textbf{Acknowledgments.}  
The author would like to thank Zhizhang Xie for helpful discussions.

\section{Preliminaries on higher indices}\label{sec:preliminary}
In this subsection, we review the higher index theory for elliptic differential operators. For more details, see \cite{willett2020higher,Yulocalization}.

\begin{definition}
	Let $X$ be a proper metric space and $\Gamma$ acts on $X$ properly and isometrically. A \emph{$\Gamma$-$X$-module} is a Hilbert space $H$ equipped with a $C_0(X)$-action $\varphi$ and a $\Gamma$-action $\pi$, such that
	\[
	\pi(\gamma)\big(\varphi(f)\xi\big)=\varphi(f^\gamma)\big(\pi(\gamma)\xi\big),\qquad \forall f\in C_0(X),~\gamma\in\Gamma,~\xi\in H,
	\]
	where $f^\gamma(x)\coloneqq f(\gamma^{-1}x)\in C_0(X)$.
\end{definition}

A $\Gamma$-$X$-module is called \emph{admissible} if
\begin{itemize}
	\item $H$ is \emph{non-degenerate}, namely the representation $\varphi$ is non-degenerate;
	\item $H$ is \emph{standard}, namely no non-zero function in $C_0(X)$ acts as a compact operator.
\end{itemize}

\begin{definition}\label{def:prop-locallycompact}
	Let $H$ be an admissible $\Gamma$--$X$-module $T$ a bounded linear operator on $H$.
	\begin{enumerate}
		\item The \emph{propagation} of $T$ is
		\[
		\prop(T)=\sup\{d(x,y)\mid (x,y)\in \supp(T)\},
		\]
		where $\supp(T)$ is the complement in $X\times X$ of the set of points $(x,y)$ for which there exist $f_1,f_2\in C_0(X)$ with $f_1Tf_2=0$ and $f_1(x)f_2(y)\neq 0$.
		\item $T$ is \emph{locally compact} if both $fT$ and $Tf$ are compact for all $f\in C_0(X)$.
		\item $T$ is \emph{$\Gamma$-equivariant} if $\gamma T=T\gamma$ for all $\gamma\in\Gamma$.
	\end{enumerate}
\end{definition}

\begin{definition}\label{def roe and localization}
	Let $H$ be an admissible $\Gamma$-$X$-module and $B(H)$ the algebra of bounded operators on $H$. The \emph{equivariant Roe algebra} of $X$, denoted $C^*(X)^\Gamma$, is the $C^*$-algebra generated by all locally compact, $\Gamma$-equivariant, finite-propagation operators in $B(H)$.
\end{definition}

For example, if $X$ is a complete Riemannian manifold and $E$ a $\Gamma$-equivariant Hermitian bundle over $X$, then $L^2(X,E)$ is a $\Gamma$-$\widetilde M$-module. In particular, if $\Gamma$ furthermore acts on $X$ \emph{cocompactly}, then we have $C^*(X)^\Gamma\cong C^*_r(\Gamma)\otimes \mathcal K$, so
\[
K_*(C^*(\widetilde M)^\Gamma)\cong K_*(C^*_r(\Gamma)).
\]
Here $C^*_r(\Gamma)$ denote the reduced group $C^*$-algebra of $\Gamma$, generated by $\Gamma$ with respect to the operator norm induced by the left regular action of $\Gamma$ on $\ell^2(\Gamma)$.

Let us now recall the definition of the higher index of elliptic operators. Assume for simplicity that $M$ is a closed Riemannian manifold, $\Gamma=\pi_1(M)$ and $X=\widetilde M$, so that $\Gamma$ acts cocompactly on $\widetilde M$. A continuous function $F\colon\R\to\R$ is a \emph{normalizing function} if it is non-decreasing, odd, and satisfies
\[
\lim_{x\to\pm\infty}F(x)=\pm 1.
\]
For instance, $F(x)=bx/\sqrt{1+b^2x^2}$ for any $b>0$.

Let $E$ be a Hermitian bundle over $M$ and $D_M$ a first-order self-adjoint elliptic operator on $E$. Let $D_{\widetilde M}$ and $\widetilde E$ denote their lifts to $\widetilde M$. Then $H=L^2(\widetilde M,\widetilde E)$ is an admissible $\Gamma$-$\widetilde M$-module.

For the $K_0$-index, assume that $E$ carries a $\mathbb Z_2$-grading and that $D_{\widetilde M}$ is odd:
\[
D_{\widetilde M}=\begin{pmatrix}0&D_{\widetilde M}^+\\ D_{\widetilde M}^-&0\end{pmatrix}.
\]

Let $\chi$ be a normalizing function. Since $\chi$ is odd, $\chi(D_{\widetilde M})$ is also odd and self-adjoint:
\begin{equation}\label{eq:chi}
	F(D_{\widetilde M})=\begin{pmatrix}0&U\\V&0\end{pmatrix}.
\end{equation}
Define
\[
W=\begin{pmatrix}1&U\\0&1\end{pmatrix}
\begin{pmatrix}1&0\\-V&1\end{pmatrix}
\begin{pmatrix}1&U\\0&1\end{pmatrix}
\begin{pmatrix}0&-1\\1&0\end{pmatrix},\qquad
e_{1,1}=\begin{pmatrix}1&0\\0&0\end{pmatrix},
\]
and set
\begin{equation}\label{eq:PD}
	P = We_{1,1}W^{-1}
	= \begin{pmatrix}
		1-(1-UV)^2 & (2-UV)U(1-VU)\\
		V(1-UV)&(1-VU)^2
	\end{pmatrix}.
\end{equation}
Then $P$ is an idempotent in $M_2((C_r^*(\Gamma)\otimes\mathcal K)^+)$, and $P-e_{1,1}\in M_2(C_r^*(\Gamma)\otimes\mathcal K)$. The higher index of $D_{\widetilde M}$ is
\[
\ind_\Gamma(D_{\widetilde M}) \coloneqq [P]-[e_{1,1}] \in K_0(C^*(\widetilde M)^\Gamma)\cong K_0(C^*_r(\Gamma)).
\]

For an ungraded operator, the $K_1$-index is
\[
\ind_\Gamma(D_{\widetilde M})=[e^{2\pi i\frac{F(D_{\widetilde M})+1}{2}}] \in K_1(C^*(\widetilde M)^\Gamma)\cong K_1(C^*_r(\Gamma)).
\]

Let $K_*(\widetilde M)^\Gamma$ denote the $\Gamma$-equivariant $K$-homology of $\widetilde M$, which is isomorphic to $K_*(M)$. Elements of $K_*(M)$ are represented by Dirac-type operators $D_M$, so we obtain the assembly map
\[
\mu=\ind_\Gamma\colon K_*(M)\to K_*(C^*(\widetilde M)^\Gamma)\cong K_*(C^*_r(\Gamma)),\qquad [D_M]\mapsto \ind_\Gamma(D_{\widetilde M}).
\]

\begin{conjecture}[Rational strong Novikov conjecture for fundamental groups of aspherical manifolds]\label{conj:Novikov}
	Let $M$ be a closed aspherical manifold and $\Gamma=\pi_1(M)$. Then $\Gamma$ satisfies the rational strong Novikov conjecture if
	\[
	\mu\colon K_*(M)\otimes\Q \to K_*(C^*_r(\Gamma))\otimes\Q
	\]
	is injective.
\end{conjecture}

The readers may refer to \cite{willett2020higher} for the strong Novikov conjecture for arbitrary finitely presented groups.

The above map $\mu=\ind_\Gamma$ indeed factors through the maximal group $C^*$-algebra $C^*_{\max}(\Gamma)$.
By \cite{MR4542729}, Theorems \ref{thm:connectedSum} and \ref{thm:main} also hold if the following maximal version of the rational strong Novikov conjecture holds.

\begin{conjecture}[Maximal rational strong Novikov conjecture for fundamental groups of aspherical manifolds]\label{conj:NovikovMaximal}
	Let $M$ be a closed aspherical manifold and $\Gamma=\pi_1(M)$. Then $\Gamma$ satisfies the maximal rational strong Novikov conjecture if
	\[
	\mu_{\max}\colon K_*(M)\otimes\Q \to K_*(C^*_r(\Gamma))\otimes\Q
	\]
	is injective.
\end{conjecture}

Now let $\widetilde M$ be a regular $\Gamma$-cover of $M$, which is complete but possibly non-compact. In this case $C^*(\widetilde M)^\Gamma$ need not coincide with $C^*_r(\Gamma)$. Let $D_M$ be a Dirac-type operator and $D_{\widetilde M}$ its lift. If $D_M$ is \emph{invertible outside a compact set}, then the higher index of $D_{\widetilde M}$ is still well defined in $K_*(C^*_r(\Gamma))$.

More precisely, let $K\subset M$ be compact and $\widetilde K$ its lift. Suppose $D_{\widetilde M}$ is invertible outside $\widetilde K$, meaning that there exists $\delta>0$ such that
\[
\int_{\widetilde M}|D_{\widetilde M}\sigma|^2\ge \delta^2\int_{\widetilde M}|\sigma|^2
\]
for all smooth sections $\sigma$ supported outside $\widetilde K$. Then the higher index class $\ind_\Gamma(D_{\widetilde M})$ defined as above lies in
\[
K_*\big(C^*(\mathcal N_R(\widetilde K))^\Gamma\big)\cong K_*(C^*_r(\Gamma)),
\]
where $\mathcal N_R(\widetilde K)$ is an $R$-neighborhood of $\widetilde K$ for some $R>0$, on which $\Gamma$ acts cocompactly. See \cite{XieYu14} for the precise definition.

\section{Non-existence of uniformly positive scalar curvature}\label{sec:UPSC}
In this section, we prove a simple version of Theorem \ref{thm:main} to illustrate the key ideas. This version follows from classical index-theoretic tools, while the proof of the full version requires additional new ingredients and will be presented in the next section.
\begin{proposition}\label{prop:noPSC}
	Let $M^n$ be a closed aspherical manifold and $\Omega\subset M$ a smooth region. Assume that
	$$\pi_1(\Omega)\to \pi_1(M)\textup{ is the trivial map}.$$
	Let $N^n$ be a spin manifold with boundary, and $i\colon \partial N\to \partial \Omega$ a diffeomorphism preserving the spin structure. Set $Z\coloneqq (M\backslash\Omega)\cup_{\partial\Omega} N$. If the rational strong Novikov conjecture holds for $\pi_1(M)$, then $Z$ does not carry any complete metric with uniform positive scalar curvature, i.e., scalar curvature bounded below by some positive constant.
\end{proposition}

\begin{proof}
	We begin by explaining the spin structure on $\Omega$. 	Denote $\Gamma=\pi_1(M)$ and let $\widetilde M$ be the universal cover of $M$. Since $\widetilde M$ is contractible, it is automatically spin. However, the $\Gamma$-action on $\widetilde M$ may not preserve a chosen spin structure, as $M$ itself may not be spin. Nevertheless, there exists a double cover $\widehat\Gamma$ of $\Gamma$ acting on $\widetilde M$ and preserving the spin structure; see \cite[Section 3.B)]{Rosenberg}. Because the map $\pi_1(\Omega)\to\pi_1(M)$ is trivial, the lift of $\Omega$ to $\widetilde M$ is the disjoint union
	$\widetilde\Omega\coloneqq\Gamma\times\Omega,$
	where $\Gamma$ acts by translation. Since the $\widehat\Gamma$-action preserves the spin structure on $\Gamma\times\Omega\subset \widetilde M$, the induced $\Gamma$-action also preserves the spin structure. Hence $\Omega$ inherits a unique spin structure from $\widetilde M$.
	
	Let $g$ be a Riemannian metric on $Z$ with uniformly positive scalar curvature, that is,
	$$\Sc_g\geq \delta>0 \quad \text{on } Z.$$
	 As above, the lift of $\Omega$ in $\widetilde M$ decomposes as $\widetilde\Omega=\Gamma\times\Omega$.
	
	Define the $\Gamma$-cover of $Z$ by
	$$\widetilde Z=(\widetilde M\setminus\widetilde\Omega)\cup_{\partial\widetilde{\Omega}}(\Gamma\times N),$$
	which is a $\Gamma$-equivariant spin surgery of $\widetilde M$ with $\widetilde N\coloneqq\Gamma\times N$ along $\partial\widetilde\Omega=\Gamma\times\partial\Omega$. Note that $\widetilde Z$ is in general not the universal cover of $Z$.
	
	Let $\widetilde g$ be the lifted metric on $\widetilde Z$. By assumption,
	$$\Sc_{\widetilde g}\geq \delta>0.$$
	Assume first that $M$ is spin, so the $\Gamma$-action preserves the spin structure on $\widetilde M$, hence $\widetilde Z$. Denote by $S(T\widetilde Z)$ its spinor bundle, and by $D_{\widetilde Z}$ the associated Dirac operator. By the Lichnerowicz formula, for any $\sigma\in C_c^\infty(\widetilde Z,S(T\widetilde Z))$,
	\begin{equation}
		\int_{\widetilde Z}\!|D_{\widetilde Z}\sigma|^2
		=\int_{\widetilde Z}\!|\nabla\sigma|^2+\frac{\Sc_{\widetilde g}}{4}|\sigma|^2.
	\end{equation}
	Thus,
	$$\|\sigma\|^2\geq \frac{\delta}{4}\|\sigma\|^2,$$
	which shows that $D_{\widetilde Z}$ is invertible.
	
	By the construction in the previous subsection, since $D_{\widetilde Z}$ is invertible outside a $\Gamma$-cocompact set (in fact an empty set), it defines a higher index
	$$\ind_\Gamma(D_{\widetilde Z})=0\in K_n(C_r^*(\Gamma)).$$
	
	Let
	$$Z_1=\Omega\cup_{\partial\Omega}N$$
	be a complete spin manifold with a metric $g_1$ satisfying
	$$g_1|_{N_1\setminus\mathcal N_r(\Omega)}=g|_{N\setminus\mathcal N_r(\partial N)}$$
	for sufficiently small $r>0$, where $\mathcal N_r(\cdot)$ denotes the $r$-neighborhood.
	
	 Define
	$$\widehat Z=\Gamma\times Z_1,$$
	the disjoint union of isometric copies of $Z_1$, equipped with the translation $\Gamma$-action and the $\Gamma$-equivariant metric $\widehat g$. The spin structure on $\widehat Z$ is induced from the spin structure on $N$ and on $\Gamma\times\Omega\subset\widetilde M$. Let $S(T\widehat Z)$ denote its spinor bundle and $D_{\widehat Z}$ its Dirac operator. Since $g_1$ has uniformly positive scalar curvature outside a compact set, it defines an index class
	$$\ind_\Gamma(D_{\widehat Z})\in K_*(C_r^*(\Gamma)).$$
	
	The manifolds $\widetilde Z$ and $\widehat Z$ are isometric outside a neighborhood of $\Gamma\times\partial\Omega$, and removing these isometric pieces and gluing the remaining parts yields $\widetilde M$. Hence, by the relative index theorem \cite[Theorem A]{XieYu14},
	$$\ind_\Gamma(D_{\widetilde Z})-\ind_\Gamma(D_{\widehat Z})=\ind_\Gamma(D_{\widetilde M})\in K_*(C_r^*(\Gamma)).$$
	Thus,
	$$0=\ind_\Gamma(D_{\widetilde M})+\ind_\Gamma(D_{\widehat Z})\in K_*(C_r^*(\Gamma)).$$
	
	We now show that this class is non-zero. Since $\widehat Z=\Gamma\times Z_1$ is a disjoint union,
	$$\ind_\Gamma(D_{\widehat Z})=\ind(D_{Z_1})\cdot[\id],$$
	where $\ind(D_{Z_1})$ is the Fredholm index of the Dirac operator on $Z_1$, and $[\id]$ is the class of the identity in $K_0(C_r^*(\Gamma))$. In particular, we have
	$$\mu([1])=[\id],$$
	where $\mu$ is the assembly map of Conjecture \ref{conj:Novikov} and $[1]$ denotes the trivial elliptic operator in $K_*(M)$. Therefore,
	$$\mu([D_M]+\ind(D_{Z_1})\cdot[1])=\ind_\Gamma(D_{\widetilde M})+\ind_\Gamma(D_{\widehat Z})=0,$$
	where $[D_M]$ is the $K$-homology class of the Dirac operator on $M$.
	
	Let
	$$\textup{ch}\colon K_*(M)\otimes\mathbb Q\longrightarrow H_*(M,\mathbb Q)$$
	be the rational Chern character, an isomorphism by a standard Mayer--Vietoris argument. By the Atiyah--Singer index theorem,
	$$
	\begin{cases}
		\textup{ch}([1])=[1]\in H_0(M,\mathbb Q),\\[2mm]
		\textup{ch}([D_M])=(\hat A(TM))^*=[M]+\textup{lower degree terms},
	\end{cases}
	$$
	where $*$ denotes Poincar\'e duality and $[M]$ is the fundamental class in $H_n(M,\mathbb Q)$. Under the composite map
	$$H_*(M,\mathbb Q)\xrightarrow{\textup{ch}^{-1}}K_*(M)\otimes\mathbb Q\xrightarrow{\mu}K_*(C_r^*(\Gamma))\otimes\mathbb Q,$$
	the class $\ind_\Gamma(D_{\widetilde M})+\ind_\Gamma(D_{\widehat Z})$ has a non-zero preimage
	$$[M]+\textup{lower degree terms}\in H_*(M,\mathbb Q).$$
	By the injectivity of $\mu$ in Conjecture \ref{conj:Novikov}, we see that
	$$\ind_\Gamma(D_{\widetilde M})+\ind_\Gamma(D_{\widehat Z})\neq 0.$$
	This contradicts the earlier equality and completes the proof.
	
	In the general case, when the $\Gamma$-translation does not preserve the spin structure on $\widetilde M$, there exists a double cover $\widehat\Gamma$ of $\Gamma$ acting on $\widetilde M$ as well as $\widetilde Z$ and $\widehat Z$ preserving the spin structures. Computing the index in $C_r^*(\widehat\Gamma)$ as in \cite[Section 3.B)]{Rosenberg} yields the same contradiction.
\end{proof}
\section{Non-negative scalar curvature on spin surgeries}\label{sec:nnsc}

In this section, we prove the remaining part of Theorem~\ref{thm:main}. The key ingredients in the proof are a concrete construction of the relative index using Callias-type operators and a quantitative Poincar\'e inequality.

To illustrate the main idea, let us first show that any complete Riemannian metric on $Z$ with non-negative scalar curvature must be scalar-flat.

\begin{proof}[Proof of Theorem \ref{thm:main}: scalar-flat part]
	Let $g$ be a Riemannian metric on $Z$ with non-negative scalar curvature. We first show that $g$ is scalar-flat. Suppose the contrary that $\Sc_g$ is positive somewhere, that is,
	\[
	\Sc_g \geq 0 \text{ on } Z, \qquad \text{and} \qquad
	\Sc_g(y)\geq \delta>0, \quad \forall y\in B_\delta(x)
	\]
	for some $x\in Z$ and $\delta>0$.
	
	We employ the same notations as in the proof of Proposition~\ref{prop:noPSC}. Let $\Gamma=\pi_1(M)$ and let $\widetilde M$ denote the universal cover of $M$. Let $\widetilde{Z}$ be the $\Gamma$-cover of $Z$ obtained by gluing $\widetilde M$ and $\Gamma\times N$, equipped with the lifted metric $\widetilde g$. Let $Z_1=\Omega\cup_{\partial\Omega} N$ be a complete spin manifold equipped with a Riemannian metric $g_1$ such that
	\[
	g_1|_{N_1\setminus \mathcal N_r(\Omega)} = g|_{N\setminus \mathcal N_r(\partial N)}.
	\]
	Let $\widehat{Z}=\Gamma\times Z_1$ be the spin manifold equipped with the $\Gamma$-translation and the $\Gamma$-equivariant metric $\widehat g$. For simplicity, we assume that $\Gamma$ preserves the spin structures on $\widetilde M$ and $\widetilde Z$.
	
	Set
	\[
	X = \widetilde{Z}\sqcup \widehat{Z}, \qquad g_X = \widetilde g \sqcup \widehat g,
	\]
	and consider the Dirac operator
	\[
	D_X = D_{\widetilde Z} \sqcup (-D_{\widehat Z})
	\]
	acting on $S(TX)$.
	
	Denote by $\widetilde N$ and $\widehat N$ the copies of $\Gamma\times N$ inside $\widetilde Z$ and $\widehat Z$, respectively. Set $Y = \widetilde N \sqcup \widehat N$. Let $\id\colon S(T\widetilde N) \to S(T\widehat N)$ be the identity map. Let $\rho\colon N \to [0,1]$ be a smooth function such that $\rho \equiv 0$ on $N_{2r}(\partial N)$ and $\rho \equiv 1$ outside $N_{3r}(\partial N)$. Extend it by zero to $Z$ and $Z_1$, and lift it to $X$, still denoted by $\rho$.
	
	Define the operator
	\[
	\Psi \coloneqq \rho U \coloneqq \rho \cdot 
	\begin{pmatrix}
		0 & \id \\
		-\id & 0
	\end{pmatrix}
	\]
	acting on $S(TX)$. More precisely, let $\sigma$ be a smooth section of $S(TX)$, and write $\rho \sigma = \widetilde{\rho\sigma} \sqcup \widehat{\rho\sigma}$ on $Y$. Then
	\[
	\Psi(\sigma) = \widehat{\rho\sigma} \sqcup -\widetilde{\rho\sigma}.
	\]
	
	Let $\varepsilon>0$ be sufficiently small. Consider
	\[
	B = D_X + \varepsilon \Psi.
	\]
	For any compactly supported smooth section $\sigma$ of $S(TX)$,
	\begin{equation}\label{eq:Lich1}
		\int_X |B\sigma|^2
		= \int_X |D\sigma|^2
		+ \varepsilon \langle (\Psi D_X + D_X \Psi) \sigma, \sigma \rangle
		+ \varepsilon^2 \rho^2 |\sigma|^2.
	\end{equation}
	
	On $Y = \widetilde N \sqcup \widehat N$, we have
	\[
	D_Y U + U D_Y = 0,
	\]
	so
	\[
	\Psi D_X + D_X \Psi = - U c(\nabla \rho) \ge -|\nabla \rho|.
	\]
	Let $m = \sup |\nabla \rho|$. The support of $\nabla \rho$ is a disjoint union of finitely many $\Gamma$-translates, which is compact modulo $\Gamma$.
	
	Using the Lichnerowicz formula, we get
	\begin{equation}\label{eq:Lichnerowicz}
		\begin{split}
			\int_X |B\sigma|^2
			&\ge
			\int_X \Bigl(|\nabla \sigma|^2 + \frac{\Sc_{g_X}}{4} |\sigma|^2 \Bigr)
			+ \varepsilon^2 \int_{\{\rho=1\}} |\sigma|^2
			- \varepsilon m \int_{\{\rho<1\}} |\sigma|^2 \\
			&\ge
			\int_X \Bigl(|\nabla \sigma|^2 + \frac{\Sc_{g_X}}{4} |\sigma|^2 \Bigr)
			+ \varepsilon^2 \|\sigma\|^2
			- (\varepsilon^2 + \varepsilon m) \int_{\{\rho<1\}} |\sigma|^2.
		\end{split}
	\end{equation}
	
	Let $B_\delta(\widetilde x)$ be the lift of $B_\delta(x)$ in $X$. Using that $\Sc_g \ge \delta$ on $B_\delta(x)$,
	\begin{equation}\label{eq:Lichnerowicz2}
		\|B\sigma\|^2
		\ge
		\int_X |\nabla \sigma|^2
		+ \frac{\delta}{4} \int_{B_\delta(\widetilde x)} |\sigma|^2
		+ \varepsilon^2 \|\sigma\|^2
		- (\varepsilon^2 + \varepsilon m) \int_{\{\rho<1\}} |\sigma|^2.
	\end{equation}
	
	Since $B_\delta(x)$ and $\{\rho<1\}$ are compact with finite distance, by \cite[Lemma~2.7]{WangXieWarped}, there exists a constant $C>0$ such that
	\[
	\int_{\{\rho<1\}} |\sigma|^2
	\le
	C \int_X |\nabla \sigma|^2
	+ C \int_{B_\delta(\widetilde x)} |\sigma|^2.
	\]
	
	Plugging this into \eqref{eq:Lichnerowicz2}, we obtain
	\[
	\|B\sigma\|^2
	\ge
	\varepsilon^2 \|\sigma\|^2
	+ (1 - C \varepsilon^2 - C m \varepsilon) \int_X |\nabla \sigma|^2
	+ \left(\frac{\delta}{4} - C \varepsilon^2 - C m \varepsilon\right) \int_{B_\delta(\widetilde x)} |\sigma|^2.
	\]
	
	If $\varepsilon>0$ is sufficiently small so that both coefficients are positive, then
	\[
	\|B\sigma\|^2 \ge \varepsilon^2 \|\sigma\|^2,
	\]
	so $B$ is invertible. However, by Lemma~\ref{lemma:index} below, $B$ defines a non-zero class in $K_*(C^*_r(\Gamma))$, which is a contradiction. This completes the proof of the scalar-flat part.
\end{proof}

Let us give the construction of the index class. We first establish an estimate for the operator $B$.

\begin{lemma}{\cite[Lemma 2.11]{WXYZlp}}\label{lemma:bound}
	Let $\chi$ be any smooth cut-off function supported in $\{\rho=1\}$. Then for any $\lambda>0$,
	\[
	\bigl\|(1+\lambda^2+b^2 B^2)^{-1/2} \rho \bigr\| \le \frac{1}{\sqrt{1+\lambda^2+b^2 \varepsilon^2}}.
	\]
\end{lemma}

\begin{proof}
	Let $\varphi$ be any smooth compactly supported spinor on $X$, and set
	\[
	\sigma = (1+\lambda^2 + b^2 B^2)^{-1/2} \varphi.
	\]
	Then
	\begin{align*}
		\int_X \langle (1+\lambda^2 + b^2 B^2) \sigma, \sigma \rangle
		&= \int_X \langle (1+\lambda^2 + b^2 D_X^2 + b^2 \varepsilon^2 \Psi^2 + b^2 \varepsilon (D_X \Psi + \Psi D_X)) \sigma, \sigma \rangle \\
		&\ge \int_X \langle (1 + \lambda^2 + b^2 \varepsilon^2 \rho^2 - b^2 \varepsilon |\nabla \rho|) \sigma, \sigma \rangle.
	\end{align*}
	
	Define
	\[
	\zeta = \max\{0, 1+\lambda^2 + b^2 \varepsilon^2 \rho^2 - b^2 \varepsilon |\nabla \rho|\}.
	\]
	Then
	\[
	\int_X \langle (1+\lambda^2+b^2 B^2) \sigma, \sigma \rangle \ge \|\sqrt{\zeta}\, \sigma\|^2.
	\]
	
	Thus
	\[
	\|\varphi\|^2 \ge \|\sqrt{\zeta} (1+\lambda^2+b^2 B^2)^{-1/2} \varphi \|^2,
	\]
	so the operator $\sqrt{\zeta} (1+\lambda^2+b^2 B^2)^{-1/2}$ has norm at most $1$. On the support of $\chi$, $\sqrt{\zeta} = \sqrt{1+\lambda^2 + b^2 \varepsilon^2}$, and therefore
	\[
	\|\rho (1+\lambda^2+b^2 B^2)^{-1/2}\| 
	= \frac{\|\rho \sqrt{\zeta} (1+\lambda^2+b^2 B^2)^{-1/2}\|}{\sqrt{1+\lambda^2 + b^2 \varepsilon^2}} 
	\le \frac{1}{\sqrt{1+\lambda^2 + b^2 \varepsilon^2}}.
	\]
	This completes the proof.
\end{proof}

\begin{lemma}	\label{lem est Fb}
	Set $F_b = F_b(B)$, where $F_b(x) = bx / \sqrt{1 + b^2 x^2}$. Then for any smooth cut-off function $\chi$ supported in $\{\rho=1\}$ with Lipschitz constant $L$, we have
\begin{enumerate}
	\item $\|[F_b, \chi]\| \le 6 \frac{L \sqrt{b}}{\sqrt{\varepsilon}}$,
	\item $\|(F_b^2 - 1)\chi\| \le \frac{1}{b \varepsilon}$.
\end{enumerate}
\end{lemma}
\begin{proof}	
	Since
\begin{equation}\label{eq:integralF}
	F_b = \frac{2}{\pi} \int_0^{+\infty} \frac{b B}{1 + \lambda^2 + b^2 B^2} \, d\lambda
\end{equation}
	in the strong operator topology, we compute:
	\[
	[F_b,\chi]
	=\frac{2}{\pi}\int_0^\infty 
	\left[
	\frac{bB}{1+\lambda^2+b^2B^2},\chi
	\right] d\lambda .
	\]
	Using the identity
	\begin{align*}
		\left[\frac{bB}{1+\lambda^2+b^2B^2},\chi\right]
		&=
		[bB,\chi](1+\lambda^2+b^2B^2)^{-1} \\
		&\quad - bB(1+\lambda^2+b^2B^2)^{-1}(bB[bB,\chi]+[bB,\chi]bB)(1+\lambda^2+b^2B^2)^{-1},
	\end{align*}
	and writing
	\[
	[B,\chi]=c(\nabla \chi)=\rho\, c(\nabla\chi)\,\rho,
	\]
	we use Lemma~\ref{lemma:bound} to estimate each term.  
	Since
	$|\nabla\chi|\leq L,$
	we obtain
	
	\begin{align*}
		&\|[bB,\chi](1+\lambda^2+b^2B^2)^{-1}\|\\
		\leq &
		b|\nabla\chi|\cdot\|\rho(1+\lambda^2+b^2B^2)^{-1/2}\|\cdot\|(1+\lambda^2+b^2B^2)^{-1/2}\|\\
		\leq& bL\cdot
		\frac{1}{\sqrt{1+\lambda^2+b^2\varepsilon^2}}\cdot\frac{1}{\sqrt{1+\lambda^2}}
		\leq \frac{L\sqrt b}{\sqrt \varepsilon}\cdot \frac{1}{(1+\lambda^2)^{3/4}},
	\end{align*}
	\begin{align*}
		&\|(1+\lambda^2+b^2B^2)^{-1}b^2B^2[bB,\chi](1+\lambda^2+b^2B^2)^{-1}\|\\
		\leq &
		b|\nabla\chi|\cdot\|(1+\lambda^2+b^2B^2)^{-1}b^2B^2\|\cdot
		\|\rho(1+\lambda^2+b^2B^2)^{-1/2}\|\cdot\|(1+\lambda^2+b^2B^2)^{-1/2}\|\\
		\leq&\frac bL\cdot 1\cdot\frac{1}{\sqrt{1+\lambda^2+b^2\varepsilon^2}}\cdot\frac{1}{\sqrt{1+\lambda^2}}\leq\frac{L\sqrt b}{\sqrt \varepsilon}\cdot \frac{1}{(1+\lambda^2)^{3/4}},
	\end{align*}
	and
	\begin{align*}
		&\|(1+\lambda^2+b^2B^2)^{-1}bB[bB,\chi](1+\lambda^2+b^2B^2)^{-1}bB\|\\
		\leq &
		b|\nabla\chi|\cdot\|(1+\lambda^2+b^2B^2)^{-1}bB\|\cdot
		\|\rho(1+\lambda^2+b^2B^2)^{-1/2}\|\cdot\|(1+\lambda^2+b^2B^2)^{-1/2}bB\|\\
		\leq& bL\cdot \frac{1}{\sqrt{1+\lambda^2}}\cdot\frac{1}{\sqrt{1+\lambda^2+b^2\varepsilon^2}}\cdot 1\leq\frac{L\sqrt b}{\sqrt \varepsilon}\cdot \frac{1}{(1+\lambda^2)^{3/4}}.
	\end{align*}
	To summarize, we obtain that
	$$\left\|\left[\frac{bB}{1+\lambda^2+b^2B^2},\chi\right]\right\|\leq \frac{3L\sqrt b}{\sqrt \varepsilon}\cdot \frac{1}{(1+\lambda^2)^{3/4}}$$
	Integrating over $\lambda\in[0,\infty)$ proves (1).
	
	For (2), we note:
	\[
	\|(1+b^2B^2)^{-1}(1-\chi)\|
	\le 
	\|(1+b^2B^2)^{-1/2}\| 
	\cdot 
	\|(1+b^2B^2)^{-1/2}\rho\|
	\le 
	\frac{1}{\sqrt{1+b^2\varepsilon^2}}
	<\frac{1}{b\varepsilon}.
	\]
\end{proof}

\begin{lemma}\label{lemma:index}
	The operator $B$ in the proof of Theorem~\ref{thm:main} defines a class $\ind_\Gamma(B) \in K_n(C^*_r(\Gamma))$. Furthermore, if the rational strong Novikov conjecture holds for $\Gamma = \pi_1(M)$, then $\ind_\Gamma(B) \ne 0$.
\end{lemma}

\begin{proof}
	The construction of $\ind(B)$ depends on the parity of the dimension of $M$.
	
	\textbf{Case 1: $\dim M$ is even.}  
	
	Pick a smooth function $\chi\colon N \to [0,1]$ such that
	\begin{itemize}
		\item $\chi$ is supported in $\{\rho=1\}$,
		\item $\chi \equiv 1$ away from the $M$-part of $Z$,
		\item $\chi$ is $L$-Lipschitz.
	\end{itemize}
	
	Write
	\[
	F_b = F_b(B) = \frac{bB}{\sqrt{1+b^2 B^2}} = 
	\begin{pmatrix} 0 & U \\ V & 0 \end{pmatrix}
	\]
	according to the natural grading on $S(TX)$. Set
	\[
	W = \begin{pmatrix} 1 & U \\ 0 & 1 \end{pmatrix}
	\begin{pmatrix} 1 & 0 \\ -V & 1 \end{pmatrix}
	\begin{pmatrix} 1 & U \\ 0 & 1 \end{pmatrix}
	\begin{pmatrix} 0 & -1 \\ 1 & 0 \end{pmatrix}, \quad e_{1,1} = \begin{pmatrix} 1 & 0 \\ 0 & 0 \end{pmatrix},
	\]
	and
	\[
	P = W e_{1,1} W^{-1} = 
	\begin{pmatrix} 1-(1-UV)^2 & (2-UV)U(1-VU) \\ V(1-UV) & (1-VU)^2 \end{pmatrix}.
	\]
	
	To obtain an element in $C^*_r(\Gamma)$, define
	\[
	\overline{P} = \sqrt{1-\chi} P \sqrt{1-\chi} + \chi e_{1,1} = e_{1,1} + \sqrt{1-\chi} (P - e_{1,1}) \sqrt{1-\chi}.
	\]
	Then $\overline{P} \in M_2((C^*_r(\Gamma) \otimes \mathcal K)^+)$ and $\overline{P} - e_{1,1} \in M_2(C^*_r(\Gamma) \otimes \mathcal K)$.  
	
	Choosing $b = a/\varepsilon$ and $L = \varepsilon / a$ in Lemma \ref{lem est Fb} for a sufficiently large universal constant $a>0$, we have $\|\overline{P} - P\| < 1/200$, ensuring
	\[
	\|\overline{P}\| \le 3, \quad \|\overline{P}^2 - \overline{P}\| \le 1/20.
	\]
	Therefore, the spectrum of $\overbar P$ is contained in $\{z\in\C:\textup{Re}(z)\ne 1/2\}$. Let $\Theta$ be the holomorphic function on $\{z\in\C:\textup{Re}(z)\ne 1/2\}$ given by
	$$\Theta(z)=\begin{cases}
		0,&\textup{Re}(z)< 1/2\\
		1,&\textup{Re}(z)> 1/2.
	\end{cases}$$
	Then $\Theta(\overbar P)$ is an idempotent in $M_2((C^*_r(\Gamma)\otimes\mathcal K)^+)$, and we define
	\[
	\ind_\Gamma(B) = [\Theta(\overline{P})] - [e_{1,1}] \in K_0(C^*_r(\Gamma)).
	\]
	
	\textbf{Case 2: $\dim M$ is odd.}  
	
	Using the same cut-off function $\chi$ and number $b$ as above, we define
	$$
	Q = e^{2\pi i \frac{F_b+1}{2}},$$
	and
	$$\overline{Q} = \sqrt{1-\chi} Q \sqrt{1-\chi} + \chi \cdot 1 = 1 + \sqrt{1-\chi} (Q-1) \sqrt{1-\chi}.$$
	Then $\overline{Q} \in (C^*_r(\Gamma) \otimes \mathcal K)^+$, with $\overline{Q}-1 \in C^*_r(\Gamma) \otimes \mathcal K$, and $\overline{Q}$ is invertible. Define
	\[
	\ind_\Gamma(B) = [\overline{Q}] \in K_1(C^*_r(\Gamma)).
	\]
	
	\vspace{.5cm}
	
As operators with finite propagation are dense in $C^*_r(\Gamma)\otimes\mathcal K$, there exists $d>0$ such that:
\begin{itemize}
	\item When $n$ is even, there is $\overbar P_{\textup{fp}}\in M_2((C^*_r(\Gamma)\otimes\mathcal K)^+)$ such that
	\begin{itemize}
		\item $\overbar P_{\textup{fp}}-\overbar P\in M_2(C^*_r(\Gamma)\otimes\mathcal K)$,
		\item $\|\overbar P_{\textup{fp}}-\overbar P\|<1/200$,
		\item the propagation of $\overbar P_{\textup{fp}}$ is no more than $d$;
	\end{itemize}
	\item When $n$ is odd, there is $\overbar Q_{\textup{fp}}\in (C^*_r(\Gamma)\otimes\mathcal K)^+$ such that
	\begin{itemize}
		\item $\overbar Q_{\textup{fp}}-\overbar Q\in C^*_r(\Gamma)\otimes\mathcal K$,
		\item $\|\overbar Q_{\textup{fp}}-\overbar Q\|<1/200$,
		\item the propagation of $\overbar Q_{\textup{fp}}$ is no more than $d$;
	\end{itemize}
\end{itemize}
Then the index class $\ind_\Gamma(B)$ above is preserved when we replace in the definition $\overbar P$ by $\overbar P_{\textup{fp}}$ (resp. $\overbar Q$ by $\overbar Q_{\textup{fp}}$).

Now we show that the index class above is non-zero. Let $\varepsilon>0$ be as in the proof of Theorem \ref{thm:main}, and $b,L,d>0$ as above. Let $N^+$ be a compact spin manifold with boundary $\partial N^+=\partial N$, such that $N^+$ is homeomorphic to $N$ when restricted to the $2d$-neighborhood of $\{\chi<1\}$. 

Set $Z^+=(M\backslash\Omega)\cup_{\partial \Omega} N^+$. Similarly, we define
\[
X^+ = \widetilde Z^+ \sqcup \widehat Z^+ 
= \bigl((\widetilde{M}\backslash\widetilde \Omega)\cup_{\partial\widetilde\Omega}(\Gamma\times N^+)\bigr) 
\sqcup (\Gamma\times Z^+),
\]
equipped with a $\Gamma$-equivariant metric $g_{X^+}$ that coincides with $g_X$ on $\mathcal N_{2d}(\{\chi<1\})$. Note that $\Gamma$ acts on $X^+$ isometrically and cocompactly. We consider the Dirac operator $D_{X^+}$ acting on $S(TX^+)$, and
\[
B_{X^+} = D_{X^+} + \varepsilon \Psi
\]
as defined in the proof of Theorem \ref{thm:main}. 

The same construction and formula as above defines the operators $\overbar P_{\textup{fp}}^+$ (resp. $\overbar Q_{\textup{fp}}^+$) and the class $\ind(B_{X^+})$ in $K_*(C^*_r(\Gamma))$. Now we observe that the operators $\overbar P_{\textup{fp}}$ and $\overbar P_{\textup{fp}}^+$ (resp. $\overbar Q_{\textup{fp}}$ and $\overbar Q_{\textup{fp}}^+$) both act on spinors over $\{\chi<1\}$ and have propagation no more than $d$, hence only depend on the geometric data on the $2d$-neighborhood of $\{\chi<1\}$. Therefore, we obtain
\[
\ind_\Gamma(B) = \ind_\Gamma(B_{X^+}) \in K_*(C^*_r(\Gamma)).
\]

Note that $\Gamma$ acts on $X^+$ cocompactly. Thus, an obvious linear homotopy connecting $B_{X^+}$ and $D_{X^+}$ shows that
\[
\ind_\Gamma(B_{X^+}) = \ind_\Gamma(D_{X^+}) \in K_*(C^*_r(\Gamma)).
\]

Now the same index-theoretical argument as in the proof of Proposition \ref{prop:noPSC} shows that, under the rational injection
\[
H_*(M,\mathbb Q) \xlongrightarrow{\textup{ch}^{-1}} K_*(M)\otimes\mathbb Q \xlongrightarrow{\ind_\Gamma} K_*(C^*_r(\Gamma))\otimes\mathbb Q,
\]
the index class $\ind_\Gamma(D_{X^+})$ admits a non-zero preimage in $H_*(M,\mathbb Q)$ given by
\[
([\hat A(TM)^*]+\ind(D_{\Omega\cup_{\partial\Omega}N^+})\cdot [1])=[M] + \textup{lower degree terms},
\]
where $\Omega\cup_{\partial\Omega} N^+$ is a closed spin manifold and $[1]\in H_0(M,\mathbb Q)$.
This finishes the proof.

\end{proof}

We now finish the proof of Theorem \ref{thm:main} by showing any such metric is Ricci-flat.  Classically, this Ricci-flatness follows from a result of Kazdan \cite{Kazdan}; see also \cite{GromovLawson}. Here, we provide a proof that depends only on the technique of Dirac operators, inspired by \cite{WangZhu}.

\begin{proof}[Proof of Theorem \ref{thm:main}: Ricci-flat part]
	Let $g$ be a complete metric on $Z$ with non-negative scalar curvature. From the first part, we see that $\Sc_g = 0$ on $Z$.  
	
	Assume that $\Ric_g$ is not identically zero. Then there exist $\delta>0$ and $x\in Z$ such that
	\[
	|\Ric_g| \ge \delta > 0\textup{ on } B_\delta(x).
	\]
	Furthermore, let $\{e_i\}_{1\leq i\leq n}$ be a local orthonormal basis near $x$. Then we may assume that
		\[
	|\Ric_g(e_1)| \ge \delta > 0 \textup{ on } B_\delta(x).
	\]
	
	Let us use the same notations as in the first part of the proof. Let $\psi$ be a smooth $\Gamma$-equivariant cut-off function supported on $B_\delta(\widetilde x)$ with $\psi \equiv 1$ on $B_{\delta/2}(\widetilde x)$. We still denote by $\{e_i\}$ its lift to $B_\delta(\widetilde{x})$, and $\nabla$ the spinorial connection. Then for any smooth compactly supported spinor $\sigma$,
	\[
	\sum_{j=1}^n c(e_1) (\nabla_{e_1}\nabla_{e_j} - \nabla_{e_j}\nabla_{e_1} - \nabla_{[e_1,e_j]}) \sigma = -\frac12 c(\Ric_{\widetilde g}(e_1)) \cdot \sigma.
	\]
	
	Hence,
	\[
	\Big\langle \sum_{j=1}^n c(e_1)(\nabla_{e_1}\nabla_{e_j}-\nabla_{e_j}\nabla_{e_1}-\nabla_{[e_1,e_j]})\sigma, -\frac12 c(\Ric_{\widetilde g}(e_1)) \cdot \psi \sigma \Big\rangle = \frac14 \psi |\Ric_{\widetilde g}(e_1)|^2 |\sigma|^2.
	\]
	
	Integrating over $B_\delta(\widetilde x)$ and applying Stokes’ theorem, we obtain
	
	\begin{align*}
		&\delta\int_{B_{\delta/2}(\widetilde x)} |\sigma|^2\leq \int_{B_{\delta}(\widetilde{x})}\frac14 \psi |\Ric_{\widetilde g}(e_1)|^2 |\sigma|^2\\
	=&\sum_{j=1}^n\int_{B_{\delta}(\widetilde{x})}\big\langle
	\nabla_{e_j}\sigma,-\frac 1 2\nabla_{e_1}\left(c(e_1)c(\Ric_{\widetilde g}(e_1))\psi \sigma\right)\big\rangle
	-\int_{B_{\delta}(\widetilde{x})}\big\langle
	\nabla_{e_1}\sigma,-\frac 1 2\nabla_{e_j}\left(c(e_1)c(\Ric_{\widetilde g}(e_1))\psi \sigma\right)\big\rangle\\
	&-\int_{B_{\delta}(\widetilde{x})}\big\langle \nabla_{[e_1,e_j]}\sigma,\frac 1 2 c(\Ric_{\widetilde g}(e_1))\psi \sigma\big\rangle\\
	\leq &~C_1 \int_{B_\delta(\widetilde x)} \big(|\nabla \sigma|^2 + 2|\langle \nabla \sigma, \sigma \rangle|\big).
	\end{align*}
	for some $C_1>0$ depending only on $n$ and $g$ on $B_\delta(x)$.
	
	By the Poincar\'e inequality \cite[Lemma 2.7]{WangXieWarped}, there exists $C_2>0$ such that
	\[
	\int_{B_\delta(\widetilde x)} |\sigma|^2 \le C_2 \int_{B_\delta(\widetilde x)} |\nabla \sigma|^2 + C_2 \int_{B_{\delta/2}(\widetilde x)} |\sigma|^2.
	\]
	
	Thus for any $\alpha>0$, we have
	\begin{equation}
		\begin{split}
			\delta^2\cdot\int_{B_{\delta/2}(\widetilde x)}|\sigma|^2\leq &~C_1\int_{B_{\delta}(\widetilde x)}\left(|\nabla\sigma|^2+\frac{1}{\alpha}|\nabla\sigma|^2+\alpha|\sigma|^2\right)\\
			\leq &\left(C_1+\frac{1}{\alpha}\right)\int_{B_{\delta}(\widetilde x)}|\nabla\sigma|^2\\
			&+\alpha C_1C_2\int_{B_{\delta}(\widetilde x)}|\nabla\sigma|^2+\alpha C_1C_2\int_{B_{\delta/2}(\widetilde x)}|\nabla\sigma|^2
		\end{split}
	\end{equation}
	Let us pick $\alpha=\delta^2/(2C_1C_2)$ and set
	$$C_3=\frac{2}{\delta^2}\left(C_1+\frac{1}{\alpha}+\alpha C_1C_2\right).$$
	We then obtain that
	\begin{equation}
		\int_{B_{\delta/2}(\widetilde x)}|\sigma|^2\leq C_3\int_{B_{\delta}(\widetilde x)}|\nabla\sigma|^2\leq C_3\|\nabla\sigma\|^2.
	\end{equation}
	Therefore, line \eqref{eq:Lichnerowicz} implies that
	\begin{equation}\label{eq:Lichnerowicz3}
		\|B\sigma\|^2\geq \frac 1 2\int_X|\nabla\sigma|^2+\frac{1}{2C_3}\int_{B_{\delta/2}(\widetilde x)}|\sigma|^2+\varepsilon^2\|\sigma\|^2-(\varepsilon^2+\varepsilon m)\int_{\{\rho<1\}}|\sigma|^2.
	\end{equation}
	Now the same proof as in the first part yields a contradiction.
\end{proof}

Now we give an example showing that the condition of the vanishing of $\pi_1(\Omega)\to\pi_1(M)$ in Theorem \ref{thm:main} is necessary.
\begin{example}\label{example}
	Let $M^n$ be a closed aspherical spin manifold with $n\ge 5$, e.g., the $n$-torus. Suppose that $\pi_1(M)$ is represented by finitely many disjoint embedded circles $S^1$ in $M$, whose tubular neighborhoods are diffeomorphic to $S^1 \times B^{n-1}$. Remove these tubular neighborhoods and attach $B^2 \times S^{n-2}$ to obtain a closed manifold $M_1$. Then $M_1$ is simply connected.  
	
	Finally, set $Z \coloneqq M_1 \# M_1$, which is a spin surgery of $M$. Note that $Z$ is simply connected, and the Dirac operator on $Z$ has vanishing $KO$-index as $Z$ is a double. Hence, $Z$ admits a metric with positive scalar curvature \cite{Stolz}.
\end{example}

\bibliographystyle{amsalpha}
\bibliography{ref.bib}

@article {MR4542729,
	AUTHOR = {Guo, Hao and Xie, Zhizhang and Yu, Guoliang},
	TITLE = {A {L}ichnerowicz vanishing theorem for the maximal {R}oe
	algebra},
	JOURNAL = {Math. Ann.},
	FJOURNAL = {Mathematische Annalen},
	VOLUME = {385},
	YEAR = {2023},
	NUMBER = {1-2},
	PAGES = {717--743},
	ISSN = {0025-5831,1432-1807},
	MRCLASS = {46L80 (53C20 58B34)},
	MRNUMBER = {4542729},
	DOI = {10.1007/s00208-021-02333-0},
	URL = {https://doi.org/10.1007/s00208-021-02333-0},
}

@article {Stolz,
	AUTHOR = {Stolz, Stephan},
	TITLE = {Simply connected manifolds of positive scalar curvature},
	JOURNAL = {Ann. of Math. (2)},
	FJOURNAL = {Annals of Mathematics. Second Series},
	VOLUME = {136},
	YEAR = {1992},
	NUMBER = {3},
	PAGES = {511--540},
	ISSN = {0003-486X,1939-8980},
	MRCLASS = {57R15 (53C21 55T15 57R90)},
	MRNUMBER = {1189863},
	MRREVIEWER = {Jonathan\ M.\ Rosenberg},
	DOI = {10.2307/2946598},
	URL = {https://doi.org/10.2307/2946598},
}

@article {Kazdan,
	AUTHOR = {Kazdan, Jerry L.},
	TITLE = {Deformation to positive scalar curvature on complete
	manifolds},
	JOURNAL = {Math. Ann.},
	FJOURNAL = {Mathematische Annalen},
	VOLUME = {261},
	YEAR = {1982},
	NUMBER = {2},
	PAGES = {227--234},
	ISSN = {0025-5831,1432-1807},
	MRCLASS = {53C20 (58G30)},
	MRNUMBER = {675736},
	MRREVIEWER = {V.\ I.\ Oliker},
	DOI = {10.1007/BF01456220},
	URL = {https://doi.org/10.1007/BF01456220},
}

@article {WangXieWarped,
	AUTHOR = {Wang, Jinmin and Xie, Zhizhang},
	TITLE = {Scalar curvature rigidity of degenerate warped product spaces},
	JOURNAL = {Trans. Amer. Math. Soc. Ser. B},
	FJOURNAL = {Transactions of the American Mathematical Society. Series B},
	VOLUME = {12},
	YEAR = {2025},
	PAGES = {1--37},
	ISSN = {2330-0000},
	MRCLASS = {53C24 (19K56 53C23)},
	MRNUMBER = {4853305},
	MRREVIEWER = {Man-Ho\ Ho},
	DOI = {10.1090/btran/206},
	URL = {https://doi.org/10.1090/btran/206},
}

@article {MR4291609,
	AUTHOR = {Deng, Jialong},
	TITLE = {Enlargeable length-structure and scalar curvatures},
	JOURNAL = {Ann. Global Anal. Geom.},
	FJOURNAL = {Annals of Global Analysis and Geometry},
	VOLUME = {60},
	YEAR = {2021},
	NUMBER = {2},
	PAGES = {217--230},
	ISSN = {0232-704X,1572-9060},
	MRCLASS = {53C23 (53C21 57N16)},
	MRNUMBER = {4291609},
	MRREVIEWER = {Mikhail\ G.\ Katz},
	DOI = {10.1007/s10455-021-09772-7},
	URL = {https://doi.org/10.1007/s10455-021-09772-7},
}

@unpublished{WXYZlp,
	author = {Jinmin Wang and Zhizhang Xie and Guoliang Yu and Bo Zhu},
	eprint = {2411.15070},
	note = {arXiv:2411.15070},
	title = {$\ell^p$-coarse Baum-Connes conjecture for $\ell^{q}$-coarse embeddable spaces},
	year = {2024}}

@unpublished{WangZhu,
	author = {Jinmin Wang and Bo Zhu},
	eprint = {2408.08245},
	note = {arXiv:2408.08245},
	title = {Sharp bottom spectrum and scalar curvature rigidity},
	year = {2024}}

@unpublished{WangXieLinf,
	author = {Jinmin Wang and Zhizhang Xie},
	eprint = {2407.21312},
	note = {arXiv:2407.21312},
	title = {Scalar curvature rigidity of spheres with subsets removed and $L^\infty$ metrics},
	year = {2024}}

@unpublished{ChenChuZhu,
	author = {Shuli Chen and Jianchun Chu and Jintian Zhu},
	eprint = {2312.04698},
	note = {arXiv:2312.04698},
	title = {Positive scalar curvature metrics and aspherical summands},
	year = {2023}}

@article {GongWuYu,
	AUTHOR = {Gong, Sherry and Wu, Jianchao and Yu, Guoliang},
	TITLE = {The {N}ovikov conjecture, the group of volume preserving
	diffeomorphisms and {H}ilbert-{H}adamard spaces},
	JOURNAL = {Geom. Funct. Anal.},
	FJOURNAL = {Geometric and Functional Analysis},
	VOLUME = {31},
	YEAR = {2021},
	NUMBER = {2},
	PAGES = {206--267},
	ISSN = {1016-443X,1420-8970},
	MRCLASS = {46L80 (46L87 58B20)},
	MRNUMBER = {4268302},
	MRREVIEWER = {Robert\ Yuncken},
	DOI = {10.1007/s00039-021-00563-7},
	URL = {https://doi.org/10.1007/s00039-021-00563-7},
}

@article {WangZhang,
	AUTHOR = {Wang, Xiangsheng and Zhang, Weiping},
	TITLE = {On the generalized {G}eroch conjecture for complete spin
	manifolds},
	JOURNAL = {Chinese Ann. Math. Ser. B},
	FJOURNAL = {Chinese Annals of Mathematics. Series B},
	VOLUME = {43},
	YEAR = {2022},
	NUMBER = {6},
	PAGES = {1143--1146},
	ISSN = {0252-9599,1860-6261},
	MRCLASS = {58J20 (53C21)},
	MRNUMBER = {4519608},
	MRREVIEWER = {Alexander\ Engel},
	DOI = {10.1007/s11401-022-0381-y},
	URL = {https://doi.org/10.1007/s11401-022-0381-y},
}

@unpublished{SuZhang,
	author = {Guangxiang Su and Weiping Zhang},
	eprint = {1705.00553},
	note = {arXiv:1705.00553},
	title = {Positive scalar curvature and connected sums},
	year = {2017}}

@unpublished{SuRemark,
		author = {Guangxiang Su},
		eprint = {2510.18329},
		note = {arXiv:2510.18329},
		title = {A remark on $\wedge^2$-enlargeable manifolds},
		year = {2025}}

@article {ChodoshLi,
	AUTHOR = {Chodosh, Otis and Li, Chao},
	TITLE = {Generalized soap bubbles and the topology of manifolds with
	positive scalar curvature},
	JOURNAL = {Ann. of Math. (2)},
	FJOURNAL = {Annals of Mathematics. Second Series},
	VOLUME = {199},
	YEAR = {2024},
	NUMBER = {2},
	PAGES = {707--740},
	ISSN = {0003-486X,1939-8980},
	MRCLASS = {53C21 (53A10)},
	MRNUMBER = {4713021},
	MRREVIEWER = {Alberto\ Roncoroni},
	DOI = {10.4007/annals.2024.199.2.3},
	URL = {https://doi.org/10.4007/annals.2024.199.2.3},
}

@article {XieYu14,
	AUTHOR = {Xie, Zhizhang and Yu, Guoliang},
	TITLE = {A relative higher index theorem, diffeomorphisms and positive
	scalar curvature},
	JOURNAL = {Adv. Math.},
	FJOURNAL = {Advances in Mathematics},
	VOLUME = {250},
	YEAR = {2014},
	PAGES = {35--73},
	ISSN = {0001-8708,1090-2082},
	MRCLASS = {58J22 (19K56 46L87 53C20 58J20)},
	MRNUMBER = {3122162},
	MRREVIEWER = {Jonathan\ M.\ Rosenberg},
	DOI = {10.1016/j.aim.2013.09.011},
	URL = {https://doi.org/10.1016/j.aim.2013.09.011},
}

@article{Gromov5fold,
	author = {Gromov, Misha},
	journal = {arXiv:2009.05332},
	title = {No metrics with Positive Scalar Curvatures on Aspherical 5-Manifolds},
	year = {2020}}

@article{GromovLawsonAnnals,
	author = {Gromov, Mikhael and Lawson, Jr., H. Blaine},
	doi = {10.2307/1971198},
	fjournal = {Annals of Mathematics. Second Series},
	issn = {0003-486X},
	journal = {Ann. of Math. (2)},
	mrclass = {53C05 (57R99)},
	mrnumber = {569070},
	mrreviewer = {N. J. Hitchin},
	number = {2},
	pages = {209--230},
	title = {Spin and scalar curvature in the presence of a fundamental group. {I}},
	url = {https://doi.org/10.2307/1971198},
	volume = {111},
	year = {1980}}

@article{SchoenYau3,
	author = {Schoen, R. and Yau, Shing Tung},
	doi = {10.2307/1971247},
	fjournal = {Annals of Mathematics. Second Series},
	issn = {0003-486X},
	journal = {Ann. of Math. (2)},
	mrclass = {58E12 (49F10 53C42)},
	mrnumber = {541332},
	mrreviewer = {Jonathan Sacks},
	number = {1},
	pages = {127--142},
	title = {Existence of incompressible minimal surfaces and the topology of three-dimensional manifolds with nonnegative scalar curvature},
	url = {https://doi.org/10.2307/1971247},
	volume = {110},
	year = {1979}}

@article{SchoenYau,
	author = {Schoen, R. and Yau, S. T.},
	doi = {10.1007/BF01647970},
	fjournal = {Manuscripta Mathematica},
	issn = {0025-2611},
	journal = {Manuscripta Math.},
	mrclass = {53C20 (58E20 83C99)},
	mrnumber = {535700},
	mrreviewer = {Paul E. Ehrlich},
	number = {1-3},
	pages = {159--183},
	title = {On the structure of manifolds with positive scalar curvature},
	url = {https://doi.org/10.1007/BF01647970},
	volume = {28},
	year = {1979}}

@article{Rosenberg,
	author = {Rosenberg, Jonathan},
	fjournal = {Institut des Hautes \'{E}tudes Scientifiques. Publications Math\'{e}matiques},
	issn = {0073-8301},
	journal = {Inst. Hautes \'{E}tudes Sci. Publ. Math.},
	mrclass = {58G12 (46L05 46M20 53C20 57R15)},
	mrnumber = {720934},
	mrreviewer = {Howard D. Rees},
	number = {58},
	pages = {197--212 (1984)},
	title = {{$C^{\ast} $}-algebras, positive scalar curvature, and the {N}ovikov conjecture},
	url = {http://www.numdam.org/item?id=PMIHES_1983__58__197_0},
	year = {1983}}

@article{GromovLawson,
	author = {Gromov, Mikhael and Lawson, Jr., H. Blaine},
	fjournal = {Institut des Hautes \'{E}tudes Scientifiques. Publications Math\'{e}matiques},
	issn = {0073-8301},
	journal = {Inst. Hautes \'{E}tudes Sci. Publ. Math.},
	mrclass = {58G10 (53C20 57N10 57R99)},
	mrnumber = {720933},
	mrreviewer = {Howard D. Rees},
	number = {58},
	pages = {83--196 (1984)},
	title = {Positive scalar curvature and the {D}irac operator on complete {R}iemannian manifolds},
	url = {http://www.numdam.org/item?id=PMIHES_1983__58__83_0},
	year = {1983}}

@book{willett2020higher,
	author = {Willett, Rufus and Yu, Guoliang},
	publisher = {Cambridge University Press},
	title = {Higher index theory},
	volume = {189},
	year = {2020}}

@article{Yu,
	author = {Yu, Guoliang},
	doi = {10.2307/121011},
	fjournal = {Annals of Mathematics. Second Series},
	issn = {0003-486X},
	journal = {Ann. of Math. (2)},
	mrclass = {57R67 (19K35 19L47 46L89)},
	mrnumber = {1626745},
	mrreviewer = {Evgeniy V. Troitski\u{\i}},
	number = {2},
	pages = {325--355},
	title = {The {N}ovikov conjecture for groups with finite asymptotic dimension},
	url = {https://doi.org/10.2307/121011},
	volume = {147},
	year = {1998}}

@article{Yulocalization,
	author = {Yu, Guoliang},
	doi = {10.1023/a:1007766031161},
	fjournal = {$K$-Theory. An Interdisciplinary Journal for the Development, Application, and Influence of $K$-Theory in the Mathematical Sciences},
	issn = {0920-3036},
	journal = {$K$-Theory},
	mrclass = {19K56 (46L80 58G12)},
	mrnumber = {1451759},
	mrreviewer = {Vicumpriya S. Perera},
	number = {4},
	pages = {307--318},
	title = {Localization algebras and the coarse {B}aum-{C}onnes conjecture},
	url = {https://doi.org/10.1023/a:1007766031161},
	volume = {11},
	year = {1997}}

\end{document}